\documentclass[a4paper,12pt,reqno]{amsart}

\usepackage{anysize}

\usepackage{pdfsync}
\usepackage[active]{srcltx}
\usepackage[english]{babel}
\usepackage{amscd}
\usepackage{amssymb}
\usepackage{amsthm}
\usepackage{amsmath}
\usepackage[latin2]{inputenc}
\usepackage{t1enc}
\usepackage{graphicx}
\usepackage{comment}
\usepackage{enumerate}
\usepackage{hyperref}
\usepackage{psfrag}
\usepackage{epsfig}
\usepackage{dsfont}
\usepackage{enumerate}
\usepackage{bbm}
\usepackage{caption}
\usepackage{subcaption}
\usepackage[usenames,dvipsnames,svgnames]{xcolor}
\usepackage{comment}

\usepackage{wrapfig}
\usepackage{float}

\newcommand{\ind}{{\alpha, \beta, \gamma}}

\newtheorem{theorem}{Theorem}[section]
\theoremstyle{plain}

\newtheorem{definition}[theorem]{Definition}

\newtheorem{lemma}[theorem]{Lemma}
\newtheorem{notation}{Notation}

\newtheorem{proposition}[theorem]{Proposition}
\newtheorem{remark}[theorem]{Remark}

\definecolor{trp}{rgb}{1,1,1}

\definecolor{red}{rgb}{1,0,.2}

\definecolor{blue}{rgb}{0,0,1}

\definecolor{rgrey}{rgb}{.8,0.4,.4}  % faint
\definecolor{grey}{rgb}{.13,.13,.13}  % almost black

\definecolor{green}{rgb}{0.0,0.4,0.2}

\newcommand*{\vect}[1]{\ensuremath{\underline{#1}}}

\newcommand{\pv}{\underline{p}}

\newcommand{\norm}[1]{\left\lVert #1 \right\rVert}
\newcommand{\der}{\mathrm{d}}
\newcommand{\abs}[1]{\left| #1 \right|}
\newcommand{\rest}[2]{\left.  {#1} \right|_{#2}}

\usepackage{fancyhdr}
\numberwithin{equation}{section}

\newcommand{\R}{\mathbb{R}}
\newcommand{\N}{\mathbb{N}}

\newcommand{\tv}{\mathbf{t}}

\newcommand{\ii}{\mathbf{i}}
\newcommand{\jj}{\mathbf{j}}

\newcommand{\hdim}{\mathrm{dim}_H}
\newcommand{\leb}{\mathcal{L}}

\newcommand{\kar}[2]{\mathbbm{1}_{ {#1}  }\left(  {#2} \right)}

\DeclareMathOperator*{\esssup}{ess\,sup}
\DeclareMathOperator*{\essinf}{ess\,inf}

\begin{document}

\pagestyle{myheadings}

\title[Dimension of self-similar measures with overlaps]{On the dimension of self-similar measures with complicated overlaps}

%%%%%%%%%%%%%%%%%%%%

\author{Bal\'azs B\'ar\'any}
\address[Bal\'azs B\'ar\'any]{Budapest University of Technology and Economics, Department of Stochastics, MTA-BME Stochastics Research Group, P.O.Box 91, 1521 Budapest, Hungary}
\email{balubsheep@gmail.com}

\author{Edina Szv\'ak}
\address[Edina Szv\'ak]{Budapest University of Technology and Economics, Department of Stochastics, P.O.Box 91, 1521 Budapest, Hungary}
\email{szvak.edina@gmail.com}

\thanks{B\'ar\'any acknowledges support from OTKA K123782, NKFI PD123970, the J\'anos Bolyai Research Scholarship of the Hungarian Academy of Sciences and the New National Excellence Program of the Ministry of Human Capacities \'UNKP-18-4-BME-385. The authors would like to thank to De-Jun Feng for the useful comments and pointing out Remark~\ref{remark}.}

\begin{abstract}
In this paper, we investigate the Hausdorff dimension of the invariant measures of the iterated function system (IFS) $\{\alpha x,\beta x,\gamma x+(1-\gamma)\}$. We provide an "almost every" type result by a direct application of the results of Feng and Hu \cite{feng2009dimension} and Kamalutdinov and Tetenov \cite{Tetenov}.
\end{abstract}

\date{\today}

\maketitle
\section{Introduction and Statement}\label{sec:intro}
Let $\mathcal{S}=\{S_1,\ldots,S_m\}$ be a family of contracing similarities on the real line, that is, $S_i(x)=r_ix+t_i$ for some  $-1 < r_i < 1$ and $t_i\in\R$. It is well known that there exists a unique, nonempty compact set $\Lambda\subset\R$ such that
$$
\Lambda=\bigcup_{i=1}^m S_i(\Lambda),
$$
see Hutchinson~\cite{Hutchinson}. We call the set $\Lambda$ {\it self-similar set} or the {\it attractor} of the iterated function system (IFS) $\mathcal{S}$. Hutchinson \cite{Hutchinson} also showed that for every probability vector $\pv=(p_1,\ldots,p_m)$ there exists a unique, compactly supported Borel probability measure $\nu$ such that
$$
\nu=\sum_{i=1}^m p_i(S_i)_*\nu.
$$
We call the measure $\nu$ {\it self-similar measure} or {\it invariant measure} with respect to the IFS $\mathcal{S}$ and probability vector $\pv$.

One of the most important topics in the field of fractal geometry is the dimension theory of self-similar sets and measures. Let us denote the Hausdorff dimension of Borel subset $A$ of $\R$ by $\dim_HA$. Moreover, let us denote the $t$-dimensional Hausdorff measure with $\mathcal{H}^t$. For the definition and basic properties of Hausdorff dimension and measure we refer the reader to \cite{Falconer2}. One can define the (lower and upper) Hausdorff dimension of a Borel measure $\mu$ on $\R$ as
\begin{eqnarray*}
\underline{\dim}_H\mu&=&\inf\{\dim_HA:\mu(A)>0\},\\
\overline{\dim}_H\mu&=&\inf\{\dim_HA:\mu(\R\setminus A)=0\}.
\end{eqnarray*}
These quantities are related to the local dimension of the measure $\mu$. Precisely, let
$$
\underline{d}_\mu(x)=\liminf_{r\to0}\frac{\log\mu(B(x,r))}{\log r}\text{ and }\overline{d}_\mu(x)=\limsup_{r\to0}\frac{\log\mu(B(x,r))}{\log r}
$$
be respectively the lower and upper {\it local dimension} of the measure $\mu$ at the point $x$, where $B(x,r)$ denotes the ball with radius $r$ and centered at $x$. Then
$$
\underline{\dim}_H\mu=\essinf_{x\sim\mu}\underline{d}_\mu(x)\text{ and }\overline{\dim}_H\mu=\esssup_{x\sim\mu}\underline{d}_\mu(x).
$$
For proofs, see \cite{Faclonerbook}. Moreover, we say that the measure $\mu$ is {\it exact dimensional} if there exists a constant $c$ such that for $\mu$-almost every $x$
$\underline{d}_\mu(x)=\overline{d}_\mu(x)=c.$ In this case, $\underline{\dim}_H\mu=\overline{\dim}_H\mu=c$, which value is denoted by $\dim_H\mu$.

It is natural to associate the points of $\Lambda$ with infinite sequences of symbols $A=\{1,\ldots,m\}$. That is, let $\Sigma=\{1,\ldots,m\}^{\N^+}$ and let $\pi\colon\Sigma\rightarrow\Lambda$ be the {\it natural projection}
$$
\pi(i_1,i_2,\ldots)=\lim_{n\to\infty}S_{i_1}\circ\cdots\circ S_{i_n}(0).
$$
Hence, $\Lambda=\pi(\Sigma)$. Moreover, if $\mu=\{p_1,\ldots,p_m\}^{\N^+}$ is the Bernoulli measure on $\Sigma$ then $\nu=\pi_*\mu=\mu\circ\pi^{-1}$ is the self-similar measure with respect to the IFS $\mathcal{S}$ and probability vector $\pv$.  We denote the left shift operator on $\Sigma$ with $\sigma$.

For a finite word $\ii=(i_1,\ldots,i_n)\in A^*$, let $S_{\ii}=S_{i_1}\circ\cdots\circ S_{i_n}$ and $r_{\ii}=\prod_{k=0}^{n}r_{i_k}$. If $\ii=(i_1, \ldots, i_n)\in A^*$, then we will use the notation $[\ii]=\{  \jj\in \Sigma : j_1=i_1, \ldots , j_n=i_n \}$. Moreover, if $\ii=(i_1, \ldots, i_n)\in A^*$ and $\jj=(j_1, j_2, \ldots )\in A^* \cup \Sigma $, then let $\ii * \jj=(i_1, i_2, \ldots, i_n, j_1, j_2, \ldots)$ be the concatenation of the words $\ii$ and $\jj$. For an $\ii=(i_1,i_2,\ldots)\in\Sigma$ and $n\geq1$, let $\ii|_n=(i_1,\ldots,i_n)$.

The Hausdorff dimension of the self-similar sets and measures is well understood if some separation holds between the functions of the IFS. Namely, we say that the IFS $\mathcal{S}=\{S_1,\ldots,S_m\}$ satisfies the {\it open set condition (OSC)} if there exists an open set $U\neq \emptyset$ such that
$$
S_i(U)\subseteq U\text{ for every $i=1,\ldots,m$ and }S_i(U)\cap S_j(U)=\emptyset\text{ for }i\neq j.
$$
Hutchinson \cite{Hutchinson} showed that under the OSC for the attractor $\Lambda$ and invariant measure $\nu$,
\[
\begin{split}
\dim_H\Lambda&=s_0,\text{ where }\sum_{i=1}^m|r_i|^{s_0}=1\text{ and }\\
\dim_H\nu&=\frac{h_\nu}{\chi_\nu},\text{ where }h_\nu=-\sum_{i=1}^m p_i\log p_i\text{ and }\chi_\nu=-\sum_{i=1}^m p_i\log|r_i|.
\end{split}
\]
The quantity $h_\nu$ is called the entropy of $\nu$ and $\chi_\nu$ is called the Lyapunov exponent of $\nu$.

The situation becomes more difficult when overlaps occur in the structure of the self-similar set. In general,
$$
\dim_H\Lambda\leq\min\{1,s_0\}\text{ and }\dim_H\nu\leq\min\{1,\frac{h_\nu}{\chi_\nu}\}.
$$

Lau and Ngai \cite{LauNgai} introduced the weak separation condition (WSC)\footnote{Although the paper was published in 1999, three years later than Zerner's result \cite{Zerner}, the preprint of the paper appeared in 1994.}, which allows exact overlaps between the cylinders but otherwise the cylinders are "well separated". That is, the identity is an isolated point of $\{S_{\ii}\circ S_{\jj}^{-1}\}_{\ii,\jj\in A^*}$. Zerner \cite{Zerner} showed that if WSC holds then $0<\mathcal{H}^{\dim_H\Lambda}(\Lambda)<\infty$. Recently, Farkas and Fraser \cite{FaFr} showed that for a self-similar set $\Lambda$ on the real line, $0<\mathcal{H}^{\dim_H\Lambda}(\Lambda)<\infty$ if and only if the corresponding IFS satisfies the WSC.

Feng and Hu \cite{feng2009dimension} showed that any self-similar measure $\nu$ is exact dimensional, regardless of separation. They also presented a formula for the dimension of the measure, which relies on the push-back of the Borel $\sigma$-algebra from $\R$ to $\Sigma$. For precise details see Section~\ref{sec:calc}. The formula is usually very hard to calculate, and according to the best knowledge of ours, it has not been used to determine the value $\dim_H\nu$ directly other than the cases of OSC or WSC. The purpose of this paper is to present an example, for which WSC does not hold, but it is possible to apply Feng and Hu's result \cite{feng2009dimension} directly.

According to our best knowledge, Pollicott and Simon \cite{PoSi} studied first a special family of parameterized self-similar IFS's, for which the WSC does not hold typically but they managed to calculate the Hausdorff dimension of the self-similar set for Lebesgue almost every parameters. In their paper, they introduced the so-called transversality condition. Later, the transversality condition had several applications in the dimension theory of iterated function systems.

We say that a parameterized family of IFS's $\mathcal{S}_\lambda=\{S_i^{(\lambda)}(x)=r_i(\lambda) x+t_i(\lambda)\}_{i=1}^m$ satisfies the {\it transversality condition} over the open set parameter space $U\subseteq\R^d$ if there exists $\delta>0$ such that for every $\ii,\jj\in\Sigma$ with $i_1\neq j_1$
\begin{equation}\label{eq:transv}
\text{if }|\pi_\lambda(\ii)-\pi_{\lambda}(\jj)|<\delta\text{ for $\lambda\in U$ then }\|\mathrm{grad}_{\lambda}(\pi_\lambda(\ii)-\pi_{\lambda}(\jj))\|>\delta.
\end{equation}

Simon and Solomyak \cite{SiSo} showed that for Lebesgue almost every $\tv=(t_1,\ldots,t_m)\in\R^m$
$$
\dim_H\Lambda_{\tv}=\min\{1,s_0\},
$$
where $\Lambda_{\tv}$ is the attractor of $\mathcal{S}_{\tv}=\{S_i(x)=r_ix+t_i\}_{i=1}^m$. Similarly, Simon, Solomyak and Urba\'nski \cite{SSU2} showed that if $\max_{i\neq j}|r_i|+|r_j|<1$ then
$$
\dim_H\nu_{\tv}=\min\{1,\frac{\sum_{i=1}^m p_i\log p_i}{\sum_{i=1}^m p_i\log|r_i|}\}\text{ for Lebesgue almost every }\tv=(t_1,\ldots,t_m)\in\R^m,
$$
where $\nu_{\tv}$ is the invariant measure w.r.t. $\mathcal{S}_\tv$ and the probability vector $\pv=(p_1,\ldots,p_m)$.

The folklore conjecture of fractal geometry claims that dimension drop ($\dim_H\Lambda<\min\{1,s_0\}$) can occur only via exact overlaps, see Hochman \cite{hochman2012self}, but it has not been verified so far. Hochman \cite[Theorem~1.1]{hochman2012self} gave a strong sufficient condition in case of self-similar systems, using methods from additive combinatorics, which ensures that there is no dimension drop. This condition is often referred as exponential separation or Diophantine condition. Precisely, we say that the {\it exponential separation condition} holds if
$$
\limsup_{n\to\infty}\frac{1}{n}\log\min_{\substack{\ii\neq\jj\in A^n\\r_{\ii}=r_{\jj}}}|S_{\ii}(0)-S_{\jj}(0)|>-\infty.
$$
Hochman \cite[Theorem~1.8]{hochman2012self} also showed the atypicality of the dimension drop in a stronger sense, namely, for any real analytic and non-degenerate parameterisations $\lambda\mapsto r_i(\lambda)$ and $\lambda\mapsto t_i(\lambda)$ for $\lambda\in U\subseteq\R$,
$$
\dim_H\{\lambda:\dim_H\Lambda_\lambda<\min\{1,s_0(\lambda)\}\}=\dim_H\{\lambda:\dim_H(\pi_\lambda)_*\mu<\min\{1,\frac{h_\nu}{\chi_\nu}\}\}=0.
$$

In this paper, we have a special interest in the following family of self-similar IFS's.
\begin{equation}\label{def:system}
\mathcal{S}_{\alpha,\beta,\gamma}=\{S_1(x)=\alpha x, S_2(x)=\beta x, S_3(x)=\gamma x+1-\gamma\},
\end{equation}
where $0<\alpha,\beta,\gamma<1$ and $\max\{\alpha,\beta\}+\gamma<1$. The system $\mathcal{S}_\ind $ has a very special structure, namely $S_1\circ S_2\equiv S_2\circ S_1$. Thus, $\mathcal{S}_\ind $ does not satisfy the transversality, nor the exponential separation condition. Also it is easy to see if $\log\alpha/\log\beta\notin\mathbb{Q}$ and $\alpha,\beta,\gamma<1/3$ then $\mathcal{S}_\ind $ does not satisfy the WSC, and thus, $\mathcal{H}^{\dim_H \Lambda_\ind }(\Lambda_\ind )=0$, see for example Fraser \cite[Section~3.1]{Fraser}.

The first author \cite{Barany} already considered the Hausdorff dimension of the attractor for every fixed value of $0<\beta,\gamma<1$ with $0<\beta+\gamma<1$ and Lebesgue almost every $\alpha\in(0,\beta)$. The proof was based on sufficiently large subsystems, which satisfy the transversality condition. Using the result of Hochman \cite{hochman2012self}, one can provide a better estimate on the exceptional set.

\begin{proposition}
	Let $\mathcal{S}_{\alpha,\beta,\gamma}$ be the IFS defined in \eqref{def:system} and let $\Lambda_{\alpha,\beta,\gamma}$ be the attractor of $\mathcal{S}_\ind $. Then for every $0<\beta,\gamma<1$  with $\beta+\gamma<1$ there exists a set $E\subset(0,\beta)$ such that $\dim_HE=0$ and for every $\alpha\in(0,\beta)\setminus E$,
	$$
	\dim_H\Lambda_{\alpha,\beta,\gamma}=\min\{1,s_1\},\text{ where $s_1$ is the unique solution of }\alpha^{s_1}+\beta^{s_1}+\gamma^{s_1}-\alpha^{s_1}\beta^{s_1}=1.
	$$
\end{proposition}

For the completeness, we give the proof here.

\begin{proof}
Let $0<\beta,\gamma<1$  with $\beta+\gamma<1$ be arbitrary but fixed. Let $\mathcal{S}^n_{\alpha,\beta,\gamma}:=\{S_{\ii}\}_{\ii\in\mathcal{C}_n}$ with attractor $\Lambda^n_{\alpha,\beta,\gamma}$, where
$$
\mathcal{C}_n=\{(\overbrace{1\cdots1}^{k}\overbrace{2\cdots2}^\ell3):k+\ell\leq n\}.
$$
By \cite[Proposition~3.2]{Barany}, for every $\varepsilon>0$ the IFS $\mathcal{S}^n_{\alpha,\beta,\gamma}$ satisfies the transversality condition \eqref{eq:transv} on $(\varepsilon,\beta-\varepsilon)$ with respect to the parameter $\alpha$. Hence, by \cite[Theorem~1.8]{hochman2012self}, there exists a set $E_n^{\beta,\gamma}\subset (0,\beta)$ with $\dim_HE_n^{\beta,\gamma}=0$ such that for every $\alpha\in(0,\beta)\setminus E_n^{\beta,\gamma}$
$$
\dim_H\Lambda^n_{\alpha,\beta,\gamma}=\min\{1,\hat{s}_n\},\text{ where }\gamma^{\hat{s}_n}\sum_{k=0}^n\sum_{\ell=0}^k\alpha^{\ell\hat{s}_n}\beta^{(k-\ell)\hat{s}_n}=1.
$$
It is easy to see that $\lim_{n\to\infty}\hat{s}_n=s_1$, where $s_1$ is defined in the statement of the proposition. Thus, for every $\alpha\in(0,\beta)\setminus\bigcup_{n=0}^{\infty}E_n^{\beta,\gamma}$
$$
\dim_H\Lambda_{\alpha,\beta,\gamma}\geq\min\{1,s_1\}.
$$
The upper bound follows by \cite[Lemma~3.5, and the Proof of Theorem~1.1]{Barany}.
\end{proof}

Unfortunately, this method does not allow us to handle the Hausdorff dimension of the corresponding self-similar measures. Recently, Kamalutdinov and Tetenov \cite{Tetenov} studied the so-called two-fold Cantor sets, which has similar structure to \eqref{def:system}. Precisely, the considered system $\{\alpha x,\beta x,\alpha x+1-\alpha,\beta x+1-\beta\}$. Based on a similar method to the transversality, Kamalutdinov and Tetenov \cite{Tetenov} showed that for Lebesgue typical parameters in $(0,1/16)^2$ the cylinders satisfy certain separation properties.

By adapting the method of Kamalutdinov and Tetenov \cite{Tetenov} with a direct application of the result of Feng and Hu \cite{feng2009dimension}, we prove the following.

\begin{theorem}\label{thm:main}
	Let $\mathcal{S}_{\alpha,\beta,\gamma}$ be the IFS defined in \eqref{def:system} and let $\nu_{\alpha,\beta,\gamma}$ be the self-similar measure with probability vector $(p_1, p_2, p_3)$ and $p_i>0$ for $i=1,2,3$. Then for every $0<\beta,\gamma<1/9$ and Lebesgue almost every $\alpha\in(0,\beta)$
	\begin{align}
	\dim_H(\nu_\ind )=\frac{-(p_1\log(p_1)+p_2\log(p_2)+p_3\log(p_3))+ \Phi(p_1, p_2, p_3)}{-(p_1\log(\alpha)+p_2\log(\beta)+p_3\log(\gamma))},
	\end{align}
	where
	\begin{align*}
	\Phi(p_1, p_2, p_3)=\sum_{k=1}^\infty  \sum_{m=1}^{k}\binom{k-1}{m-1}p_3\log\left(\frac{m}{k}\right)(p_1^mp_2^{k-m}+p_1^{k-m}p_2^m).
	\end{align*}
\end{theorem}

The structure of the paper is as follows: in Section~\ref{sec:pre} we define a certain separation condition (forward separation), and state some geometric lemmas. In Section~\ref{sec:calc}, we apply Feng and Hu's formula \cite[Theorem~2.8]{feng2009dimension} directly to calculate the dimension of the invariant measures under forward separation. Finally, using the General Positioning Theorem of Kamalutdinov and Tetenov \cite[Theorem~14]{Tetenov}, we show that almost every parameters verify the forward separation property in Section~\ref{sec:exist}.

\section{Preliminaries}\label{sec:pre}

First, let us state some basic properties of the IFS $\mathcal{S}_{\alpha,\beta,\gamma}$ and the attractor $\Lambda_{\alpha,\beta,\gamma}$. For that, we adapt the Section~1.2 of \cite{Tetenov} to our case and coin the property (i) of \cite[Proposition~2]{Tetenov} by a term {\it forward separation}. Let $L_{\alpha, \beta, \gamma}=S_1(\Lambda_{\alpha,\beta,\gamma})\cup S_2(\Lambda_{\alpha,\beta,\gamma})$ and $R_{\alpha, \beta, \gamma}=S_3(\Lambda_{\alpha,\beta,\gamma})$. It is easy to see that $\Lambda_{\alpha,\beta,\gamma}=L_{\alpha, \beta, \gamma}\cup R_{\alpha, \beta, \gamma}$.

\begin{lemma}\label{098}
	For $0<\alpha,\beta,\gamma<\frac{1}{2}$, we have
	\begin{enumerate}
		\item for all $ i \in \{  1, 2  \} $ and every $ m, n \in \mathbb{N} $ with $ m\neq n $, $ S_i^m(R_{\alpha, \beta, \gamma})\cap S_i^n(R_{\alpha, \beta, \gamma})=\emptyset$,
		\item\label{it:bou} for all $ m, n \in \mathbb{N}$, $ S_1^mS_2^n(\Lambda_{\alpha,\beta,\gamma})\subseteq S_1^m(\Lambda_{\alpha,\beta,\gamma})\cap S_2^n(\Lambda_{\alpha,\beta,\gamma})$,
		\item $ \Lambda_{\alpha,\beta,\gamma}\backslash \{  0 \} ={ \displaystyle  \bigcup_{n, m=0}^\infty } S_1^m S_2^n (R_\ind) $ .
	\end{enumerate}
\end{lemma}

\begin{proof}
	\begin{enumerate}
	\item We prove only for $i=1$, the case $i=2$ is similar. Let $m, n \in \mathbb{N}\quad m > n$. Since $\gamma<1/2$, $R_{\alpha, \beta, \gamma}\subseteq (\frac{1}{2}, 1)$, and thus $S_1^m(R_{\alpha, \beta, \gamma})\subseteq (\frac{1}{2}\alpha^m, \alpha^m)$ and $S_1^n(R_{\alpha, \beta, \gamma})\subseteq (\frac{1}{2}\alpha^n, \alpha^n)$. One can see that the right endpoint of one interval is smaller than the left endpoint of the other interval, that is, $\alpha^m=\alpha \cdot \alpha^{m-1}<\frac{1}{2}\alpha^{m-1}\leq \frac{1}{2}\alpha^{n}$.
		\item Let $m, n \in \mathbb{N}$, then $S_1^m(\Lambda_\ind )\subseteq \Lambda_\ind$ and $S_2^n(\Lambda_\ind )\subseteq \Lambda_\ind $. So, we conclude that $S_2^nS_1^m(\Lambda_\ind )\subseteq S_2^n (\Lambda_\ind )$ and $S_1^m S_2^n (\Lambda_\ind )\subseteq S_1^m (\Lambda_\ind )$. Using commutativity we can get the statements.
		\item It is easy to see that
		\begin{align}
		\pi_\ind^{-1}({\displaystyle \bigcup_{m, n=0}^\infty} S_1^mS_2^n(R_\ind))=\{ \ii\in \Sigma : \text{ there exists $k\geq1$ such that } i_k=3     \}.
		\end{align}
		For those $\ii\in \Sigma$ such that there is no $k$ for which $i_k=3$, then the image of $\ii$ is 0.
	\end{enumerate}
\end{proof}

Now, we introduce a separation property, which allows us to calculate the dimension of invariant measures.

\begin{definition}\label{080}
	We call the system $\mathcal{S}_\ind$ forward separated, if $\alpha, \beta, \gamma \in (0, \frac{1}{9})$ and
	\begin{align}
	\text{for every } m, n \in \mathbb{N}\quad m, n>0 \quad S_1^m(R_\ind)\cap S_2^n(R_\ind )=\emptyset.
	\end{align}
\end{definition}

\begin{notation}
We denote the disjoint union with $\sqcup$.
\end{notation}

\begin{lemma} \label{099}
	The system $\mathcal{S}_\ind$ is forward separated if and only if
	\begin{align}\label{eq:disun}
	\Lambda_\ind \backslash\{  0\}=\bigsqcup_{n, m=0}^\infty S_1^mS_2^n (R_\ind).
	\end{align}
\end{lemma}

\begin{proof}
	First, we assume that $\mathcal{S}_\ind$ is forward separated. Let $(m_1, n_1)\neq(m_2, n_2)$, then
	\begin{align}
	\begin{split}
	S_1^{m_1}S_2^{n_1}(R_\ind )=S_1^{min\{ m_1, m_2 \}  }S_2^{ min\{ n_1, n_2  \}  }(S_1^{k_1}S_2^{l_1}(R_\ind ))\\
	S_1^{m_2}S_2^{n_2}(R_\ind )=S_1^{min\{ m_1, m_2 \}  }S_2^{ min\{ n_1, n_2  \}  }(S_1^{k_2}S_2^{l_2}(R_\ind))
	\end{split}
	\end{align}
	hold.
	At least one of $k_1, k_2$ is zero and one of $l_1, l_2$ is zero. Thus, by the definition the forward separated property, we get that \eqref{eq:disun} holds.
	
	On the other hand, assume that $\Lambda_\ind \backslash\{  0\}=\bigsqcup_{n, m=0}^\infty S_1^mS_2^n (R_\ind)$. Then we can get that $\mathcal{S}_\ind$ is forward separated by using the conditon for the indeces $(m, 0)$ and $(0, n)$.
\end{proof}

\begin{comment}

\begin{lemma}
	If the system $\mathcal{S}_\ind$ is forward separated, then for every
	\begin{align}
	m, n \in \mathbb{N} \quad S_1^m (\Lambda_\ind)\cap S_2^n(\Lambda_\ind)=S_1^mS_2^n(\Lambda_\ind).
	\end{align}
\end{lemma}

\begin{proof}
	Using Lemma~\ref{099}, we get
	\begin{align*}
	\begin{split}
	S_1^m(\Lambda_\ind)\cap S_2^n(\Lambda_\ind )\setminus \{  0\}&=S_1^m(\Lambda_\ind\setminus \{  0\})\cap S_2^n(\Lambda_\ind\setminus \{  0\})\\
	&=S_1^m(\bigsqcup_{k, l=0}^\infty S_1^kS_2^l(R_\ind ))\cap S_2^n(\bigsqcup_{k, l=0}^\infty S_1^kS_2^l(R_\ind ))\\
	&\subseteq\bigsqcup_{k, l =0}^\infty  S_1^{k+m}S_2^{l+n}(R_\ind ) =S_1^mS_2^n(\bigsqcup_{k, l=0}^\infty S_1^kS_2^l(R_\ind))\\
	&=S_1^mS_2^n(\Lambda_\ind)\setminus \{  0\}.
	\end{split}
	\end{align*}
	The other inclusion follows by Lemma~\ref{098}\eqref{it:bou}.
\end{proof}

\end{comment}

\section{Feng and Hu's formula under forward separation}\label{sec:calc}

In this section, we describe Feng and Hu's formula \cite[Theorem~2.8]{feng2009dimension} in details. First, we need some basic properties of the conditional measures and conditional expectations.

\subsection{Conditional measures and expectations}

Let $Z$ be a compact metric space. We consider the probability space $(Z, \mathcal{B}, \mu)$, where $\mathcal{B}$ is the Borel $\sigma$-algebra of $Z$ and $\mu$ is a probability measure on $Z$.

\begin{definition}
Let $(Z, \mathcal{B}, \mu )$ be a probability space as above and let $\mathcal{G}\subseteq \mathcal{B}$ be an arbitrary sub-$\sigma$-algebra. If $ \varphi\in L^1(Z, \mathcal{B}, \mu )$, then the function $\psi\in L^1(Z, \mathcal{B}, \mu )$ is the conditional expectation of $\varphi$ with respect to the $\sigma$-algebra $\mathcal{G}$, if
	\begin{enumerate}
		\item $\psi$ is $\mathcal{G}$-measurable,
		\item for every $G\in \mathcal{G} $
		\begin{align}
		\int\displaylimits_{Z}\varphi(x) \kar{G}{x}\der \mu(x)=\int\displaylimits_{Z} \psi(x)\kar{G}{x}\der \mu(x).
		\end{align}
	\end{enumerate}
\end{definition}

\begin{theorem}
	Let $\mathcal{G}\subseteq \mathcal{B}$ be an arbitrary $\sigma$-algebra.
	If $\psi$ and $\tilde{\psi}$ are conditional expectations of the function $\varphi\in L^1(Z, \mathcal{B}, \mu)$ with respect to  $\mathcal{G}$, then $\psi(z)=\tilde{\psi}(z)$ for $\mu$-almost every $z\in Z$.
\end{theorem}

We denote the conditional expectation of $\varphi\in L^1(Z, \mathcal{B}, \mu)$ with respect to $\mathcal{G}$ with $\mathbb{E}_\mu (\varphi|\mathcal{G})$.

For a collection $B$ of subsets of $Z$, denote $\sigma(B)$ the generated $\sigma$-algebra by the set $B$. If $\mathcal{A}_i\subseteq \mathcal{B}$ are $\sigma$-algebras for all $i=1, 2, \dots $, then denote $\displaystyle \bigvee_{i=1}^\infty \mathcal{A}_i$ the common refinement of the $\sigma$-algebras, that is $\displaystyle \bigvee_{i=1}^\infty \mathcal{A}_i=\sigma\left(\displaystyle \bigcup_{i=1}^\infty \mathcal{A}_i\right)$.

We call $\mathcal{P}\subseteq \mathcal{B}$ a partition of $Z$, if for every $P_1\neq P_2\in \mathcal{P}$, $P_1\cap P_2=\emptyset$ and $\displaystyle \bigcup_{P\in \mathcal{P}}P=Z$. For $z\in Z$, the set $\mathcal{P}(z)$ denotes the unique $\mathcal{P}(z)\in \mathcal{P}$ such that $z\in \mathcal{P}(z)$.

Let $\mathcal{F}$ be a $\sigma$-algebra such that there exists some $E_1, E_2, \ldots \in\mathcal{B} $ for which
\begin{align}\label{088}
\mathcal{F}=\bigvee_{i=1}^\infty  \{\emptyset, E_i, Z/E_i, Z   \}.
\end{align}
For every $n=1,2, \dots$  let $\mathcal{P}_n$ be a partition of $Z$ such that
\begin{align}
\sigma(\mathcal{P}_n)=\bigvee_{i=1}^n \{\emptyset, E_i, Z/E_i, Z\},
\end{align}

\begin{definition}
	The set $\{ \mu_z  \}_{z\in Z}$ of Borel probability measures on $Z$ is a system of conditional measures of $\mu$ with respect to the $\sigma$-algebra $\mathcal{F}$, if
\begin{enumerate}
		\item $\text{ for every } E\in \mathcal{F}\text{,  }\,  z\in E \quad     \mu_z(E)=1$ holds for $\mu$-almost every $z\in Z$,
		\item $\text{ for every bounded measurable function }\varphi: Z\rightarrow \R$ the function $z\mapsto \displaystyle\int\displaylimits_{Z} \varphi(x) \, \der \mu_z(x)$ is $\mathcal{F} $-measurable and
\begin{align}
		\displaystyle\int\displaylimits_{Z} \varphi(x)   \,  \der \mu(x) = \displaystyle\int\displaylimits_{Z}\displaystyle\int\displaylimits_{Z} \varphi(x) \, \der \mu_z(x) \, \, \der \mu(z).
\end{align}
\end{enumerate}
\end{definition}

\begin{theorem}\label{076}
	If $\{\mu_z   \}_{z\in Z}$ and $\{ \nu_z  \}_{z\in Z}$ are two systems of condtional measures of $\mu$ with respect to $\mathcal{F}$, then $\mu_z=\nu_z$ for $\mu$-almost every $z\in Z
	$.
\end{theorem}
The proof is in \cite{rokhlin}.

\begin{theorem}\label{078}
	The limit of the measures
	\begin{align}
	\mu_z^\mathcal{F}=\lim_{n\to \infty }\frac{\rest{\mu}{\mathcal{P}_n(z)}}{\mu(\mathcal{P}_n(z))}\text{ exists for }\mu\text{-almost every }z\in Z,
	\end{align}
	where the limit is meant in the weak-star topology.
	\par
	Moreover, the set $\{  \mu_z^\mathcal{F} \}_{z\in Z}$ is a system of conditional measures of $\mu$ with respect to the $\sigma$-algebra $\mathcal{F}$.
\end{theorem}
The proof can be found in \cite{simmons}.

\begin{theorem}\label{077}
	Let $\{ \mu_z \}_{z\in Z}$ be a system of conditional measures of $\mu$ with respect to $\mathcal{F}$. Let $\varphi: Z\rightarrow \R$ is bounded and measurable, then the function
	\begin{align}
	\begin{split}
	\psi:  Z&\rightarrow \R \text{ for which }\\
	\psi(z)=\int\displaylimits_{Z}\varphi(x)\der \mu_z(x)&\text{ for }\mu\text{-almost every }z\in Z
	\end{split}
	\end{align}
	is the conditional expectation of $\varphi$ with respect to $\mathcal{F}$, thus $\mathbb{E}_\mu(\varphi|\mathcal{F})=\psi$.
\end{theorem}

The proof is in \cite{simmons}.

\subsection{Dimension of self-similar measures}

First, we introduce some notations in general, which will be used later for our specific case.

Let $(Z, \mathcal{B}, \mu)$ be as in the previous section. Let $\xi\subseteq \mathcal{B}$ be a countable partition of $Z$. Let $\mathcal{A}\subseteq \mathcal{B}$ be an arbitrary $\sigma$-algebra. Then $I_\mu(\xi|\mathcal{A})$ denotes the {\it conditional information} of the partition $\xi$ given by $\mathcal{A}$, which is
\begin{align}\label{085}
I_\mu(\xi|\mathcal{A})(x)=-\sum_{E\in \xi}\kar{E}{x} \log [\, \mathbb{E}_\mu (\mathbbm{1}_E | \mathcal{A})(x)\, ].
\end{align}

The {\it conditional entropy} of $\xi$ given $\mathcal{A}$ is defined by the following formula
\begin{align}\label{084}
H_\mu(\xi| \mathcal{A})=\int\displaylimits_Z I_\mu(\xi|\mathcal{A})\der \mu.
\end{align}

Now, we consider a self-similar IFS $\mathcal{S}=\{ S_1, \ldots , S_m\}$ on the interval $[0, 1]$. The symbolic space is $\Sigma=\{  1, \ldots, m  \}^{\mathbb{N}^+}$ and the attractor of $\mathcal{S}$ is $\Lambda$. We denote the natural projection of $\mathcal{S}$ with $\pi$. We study the probability space $(\Sigma, \mathcal{C}, \mu)$, where $\mathcal{C}$ is the $\sigma$-algera generated by the cylinder sets, and $\mu$ is a $\sigma$-invariant probability measure.
We define the measure $\nu$ on $\Lambda$ with $\nu= \pi_*\mu=\mu\circ\pi^{-1}$. The set $\mathcal{P}=\{ [1], \dots, [m]   \}$ is a Borel partition of $\Sigma$ and $\gamma$ denotes the Borel $\sigma$-algebra on $[0, 1]$. We define the {\it projection entropy} of $\mu$ under $\pi$ to the IFS $\mathcal{S}$ as
\begin{align}
h_\pi (\sigma, \mu)=H_\mu(\mathcal{P}|\sigma^{-1}\pi^{-1}\gamma)-H_\mu(\mathcal{P}|\pi^{-1}\gamma).
\end{align}

We state here a special case of \cite[Theorem~2.8]{feng2009dimension} for self-similar systems.

\begin{theorem}[Feng-Hu]\label{086}
	Let $\mu$ be a $\sigma$-invariant, ergodic Borel probabilty measure on $\Sigma$, and let $\mathcal{S}$ and $\pi$ be as above and let $\nu=\pi_*\mu=\mu\circ \pi^{-1}$. Then
	\begin{align}
	d_\nu (x)=\frac{h_\pi(\sigma, \mu )}{-\sum_{i=1}^m\mu([i])\log|r_i|}\text{ for }\nu\text{-almost every }x\in \Lambda.
	\end{align}
\end{theorem}

In order to calculate $h_\pi(\sigma,\mu)$, we need the following two lemma.

\begin{lemma}\label{083}
	Let $\mathcal{S}=\{ S_1, \dots, S_m\}$ be an IFS on $[0, 1]$. If $\mu$ is a Bernoulli measure on $\Sigma$ for the probability vector $\vect{p}=(p_1, \dots, p_m)$, then
	\begin{align}
	H_\mu(\mathcal{P}|\sigma^{-1}\pi^{-1}\gamma)=-\sum_{k=1}^m p_k \log(p_k) .
	\end{align}
\end{lemma}
\begin{proof}
	Let $[k]\in \mathcal{P}$, then the generated $\sigma$-algebra by the function $\mathbbm{1}_{[k]}$ is\linebreak $\sigma(\mathbbm{1}_{[k]})=\{\emptyset, [k], \Sigma/[k], \Sigma \}\subseteq \sigma(\mathcal{P})$, where $\sigma(\mathcal{P})$ is the generated $\sigma$-algebra by $\mathcal{P}$. We can easily check that $\sigma(\mathcal{P})$ and $\sigma^{-1}\pi^{-1}\gamma$ are independent $\sigma$-algebras. We know that if $\varphi\in L^1(\Sigma, \mathcal{C}, \mu)$, $\sigma(\varphi)$ and $\mathcal{G}$ are independent $\sigma$-algebras, then $\mathbb{E}_\mu(\varphi|\mathcal{G})=\mathbb{E}_\mu(\varphi)=\int \varphi \der \mu $. Thus $\mathbb{E}_\mu(\mathbbm{1}_{[k]}|\sigma^{-1}\pi^{-1}\gamma)=\mu([k])=p_k$. So by using \eqref{084} and \eqref{085} we can conclude that $H_\mu(\mathcal{P}|\sigma^{-1}\pi^{-1}\gamma)=-\displaystyle\sum_{k=1}^mp_k\log(p_k)$.
\end{proof}

It is well known that the $\sigma$-algebra $\gamma$ is generated by countable many finite partitions, that is, let for $n=1, 2, \ldots$
\begin{align}
\mathcal{Q}_n=\left\lbrace \left[\frac{k}{2^n}, \frac{k+1}{2^n}   \right)    : 0\leq k\leq 2^n-1\right\rbrace,
\end{align}
then $\gamma=\displaystyle\bigvee_{i=1}^\infty\sigma(\mathcal{Q}_i)$ is the Borel $\sigma$-algebra on $[0,1]$.

\begin{lemma}\label{lem:int}
	Let $\mathcal{S}=\{ S_1, \dots, S_m\}$ be an IFS on $[0, 1]$. If $\mu$ is a Bernoulli measure on $\Sigma$ for the probability vector $\vect{p}=(p_1, \dots, p_m)$, then
	\begin{align}
	H_\mu(\mathcal{P}|\pi^{-1}\gamma)=-\int\displaylimits_\Sigma \log(\mu_\ii ([i_1]))\der \mu(\ii),
	\end{align}
	where
	\begin{align}\label{eq:limprop}
	\mu_\ii([i_1])=\lim_{n\to \infty} \frac{\mu(\pi^{-1}(\mathcal{Q}_n(\pi(\ii)))\cap [i_1])}{\mu(   \pi^{-1}(\mathcal{Q}_n(\pi(\ii)))  )}\text{ for $\mu$-a.e. $\ii$}.
	\end{align}
\end{lemma}

%%%%%%%%%%%%%%%%%%%%%%%%%%%%%%%%%%%%%%%%%%%%%%%%%%
\begin{proof}
	Using Theorem \ref{076} and Theorem \ref{078}, we get that a system of conditional measures $\{ \mu_\ii  \}_{\ii\in \Sigma}$ with respect to the $\sigma$-algebra $\pi^{-1}\gamma$ exists and unique up to a set of zero measure. Then by using Theorem \ref{077},
	\begin{align}
	\mathbb{E}_\mu (\mathbbm{1}_{[k]}|\pi^{-1}\gamma)(\ii)=\int\displaylimits_\Sigma \kar{[k]}{\jj}\der \mu_\ii (\jj)=\mu_\ii([k]).
	\end{align}
	Thus, by the definition of the conditional entropy \eqref{085} and \eqref{084}, we get
\begin{align}
\begin{split}
H_\mu(\mathcal{P}|\pi^{-1}\gamma)&=-\sum_{k=1}^m\int\displaylimits_\Sigma \mathbbm{1}_{[k]}(\ii)\log\mu_\ii([k])\der\mu(\ii)\\
&=-\int\displaylimits_\Sigma \log(\mu_\ii ([i_1]))\der \mu(\ii).
\end{split}
\end{align}
By Theorem \ref{078}, the weak-star limit of the sequence
$\frac{\rest{\mu}{\pi^{-1}(\mathcal{Q}_n(\pi(\ii)))}}{\mu(   \pi^{-1}(\mathcal{Q}_n(\pi(\ii)))  )}$
exists for $\mu$-a.e. $\ii$ and equals to $\mu_\ii$. Since $[k]\subseteq \Sigma$ is open and closed,
$$
\mu_\ii([i_1])=\lim_{n\to \infty} \frac{\mu(\pi^{-1}(\mathcal{Q}_n(\pi(\ii)))\cap [i_1])}{\mu(   \pi^{-1}(\mathcal{Q}_n(\pi(\ii)))  )}.
$$\end{proof}

\subsection{Entropy under forward separation}

Now, we calculate the integral in Lemma~\ref{lem:int} under forward separation.

\begin{proposition}\label{prop:int}
Let $\mathcal{S}_\ind=\{S_1, S_2, S_3   \}$ be a forward separated system. (See Definition \ref{080}.) Let $\mu=(p_1, p_2, p_3)^{\mathbb{N}^+}$ be a Bernoulli measure on $\Sigma$ for the probability vector $\vect{p}=(p_1, p_2, p_3)$, and let $\{\mu_{\ii}\}_{\ii \in \Sigma}$ be the family of conditional measures of $\mu$ with respect to $\pi_\ind^{-1}\gamma$. Then
$$
\int\displaylimits_\Sigma\log\mu_\ii([i_1])d\mu(\ii)=\sum_{k=1}^\infty  \sum_{m=1}^{k}\binom{k-1}{m-1}p_3\log\left(\frac{m}{k}\right)(p_1^mp_2^{k-m}+p_1^{k-m}p_2^m).
$$
\end{proposition}

\begin{proof}
We use simpler notation for the mathematical objects which belongs to the IFS $\mathcal{S}_\ind$.
We denote the attractor of $\mathcal{S}_\ind$ with $\Lambda$. The natural projection is $\pi$.

	We define for all $m, n=0, 1, \dots$ the set
	$$
		E(m,n)=\{  (i_1, \dots , i_{m+n+1})\, :\, i_{m+n+1}=3, \abs{ \{ k: i_k=1  \} }=m, \, \, \abs{ \{ k:  i_k=2  \}  }=n    \}
	$$
	and
	\begin{align}
	H(m, n)=\left\lbrace  \ii=(i_1, i_2, \dots )\in \Sigma\,  : \ii|_{n+m+1}\in E(m,n)\right\rbrace .
	\end{align}
	It is easy to see that $\pi(H(m, n))=S_1^mS_2^nS_3(\Lambda)\subseteq [a(m, n), 1]$, where $a(m, n)=(1-\gamma)(\min\{ \alpha, \beta  \})^{m+n}$.
	
	Let $\ii\in H(m,n)$ be arbitrary but fixed. Then $\pi(\ii)\neq0$ thus, there exists $N\geq1$ such that $0\notin\mathcal{Q}_\ell(\pi(\ii))$ for every $\ell\geq N$. So if $p, q\in \mathbb{N}$ are such that $(\max\{\alpha, \beta  \})^{p+q}<2^{-(N+1)}$, then $S_1^pS_2^qS_3(\Lambda)\cap Q_\ell(\pi(\ii))=\emptyset$ for every $\ell\geq N$.
	
	Let
	\begin{align}
	\mathcal{H}=\left\lbrace   S_1^kS_2^lS_3(\Lambda)\, :\, (\max\{\alpha, \beta  \})^{k+l}\geq2^{-(N+1)}, \quad k,l= 0, 1,\dots   \right\rbrace .
	\end{align}
	Note that $S_1^mS_2^nS_3(\Lambda)\in \mathcal{H}$. By Lemma~\ref{099}, the set $\mathcal{H}$ is a finite collection of disjoint compact sets, thus there exists $\varepsilon_1>0$ such that
	\begin{align}
	\text{for every } H_1\neq H_2 \in \mathcal{H}\quad B_{\varepsilon_1}(H_1)\cap B_{\varepsilon_1}(H_2)=\emptyset,
	\end{align}
	where $B_\varepsilon(H)$ means the $\varepsilon$ neighbourhood of the set $H$.

Thus, by choosing $N'=\max\{N,-\log\frac{\varepsilon_1}{2}\}$, $Q_\ell(\pi(\ii))\cap \Lambda\subseteq S_1^m S_2^n (S_3(\Lambda))$ for every $\ell\geq N'$, and hence $\pi^{-1}(\mathcal{Q}_\ell(\pi(\ii))\subseteq H(m, n)$. More precisely,
	\begin{align*}
	\pi^{-1}(\mathcal{Q}_\ell(\pi(\ii)))&=E(m,n)\times T_\ell(i_{m+n+2}, i_{m+n+3}, \ldots )=E(m,n)\times T_\ell.
	\end{align*}
	Indeed, if $\jj\in\sigma^{n+m+1}\pi^{-1}(\mathcal{Q}_\ell(\pi(\ii)))$, then for any $\mathbf{k}\in E(m,n)$, $\mathbf{k}*\jj\in\pi^{-1}(\mathcal{Q}_\ell(\pi(\ii)))$. Since $\mu$ is a Bernoulli measure
	\begin{align}
	\frac{\mu(E(m,n)\times T_\ell\cap [i_1])}{\mu(E(m,n)\times T_\ell)}=\frac{\mu(H(m,n)\cap [i_1] )\mu(T_\ell)}{\mu(H(m,n))\mu(T_\ell)}=\frac{\mu(H(m,n)\cap [i_1] ) }{\mu(H(m,n))}.
	\end{align}
	On the other hand,
	\begin{align}
	\mu(H(m,n))&=\frac{(m+n)!}{m!n!}p_1^m p_2^n p_3,\\
	\mu(H(m,n)\cap [1])&=\frac{(m+n-1)!}{(m-1)!n!}p_1^m p_2^n p_3\text{ and }\\
	\mu(H(m,n)\cap [2])&=\frac{(m+n-1)!}{m!(n-1)!}p_1^m p_2^n p_3.
	\end{align}
	Thus, for every $\ii\in H(m,n)$ and every sufficiently large $\ell$
	\begin{align}
	\frac{\mu(\pi^{-1}(\mathcal{Q}_\ell(\pi(\ii)))\cap[i_1])}{\mu(\pi^{-1}(\mathcal{Q}_\ell(\pi(\ii))))}=\begin{cases}
	\frac{m}{m+n} & i_1=1,\\
	\frac{n}{m+n} & i_1=2.
	\end{cases}
	\end{align}
	
	If $\ii\in H(0,0)=[3]$ then for large enough $\ell$ we get $\pi^{-1}(\mathcal{Q}_\ell(\pi(\ii)))\subseteq[3]$, and thus
	\begin{align}
	\mu_\ii([i_1])=\lim_{\ell\to\infty}\frac{\mu(\pi^{-1}(\mathcal{Q}_\ell(\pi(\ii)))\cap[i_1])}{\mu(\pi^{-1}(\mathcal{Q}_\ell(\pi(\ii))))}=1.
	\end{align}
	Since $\mu(\bigcup_{m, n=0}^\infty H(m, n))=1$, the integral that we want to calculate is
$$
	\int\displaylimits_\Sigma  \log (\mu_\ii ([i_1]))\der \mu(\ii)=\sum_{k=1}^\infty  \sum_{m=1}^{k}\binom{k-1}{m-1}p_3\log\left(\frac{m}{k}\right)(p_1^mp_2^{k-m}+p_1^{k-m}p_2^m).
	$$
\end{proof}

Summarizing the above.

\begin{theorem}\label{thm:dimforsep}
Let $\mathcal{S}_\ind=\{S_1, S_2, S_3   \}$ be a forward separated system. (See Definition \ref{080}.) Let $\mu=(p_1, p_2, p_3)^{\mathbb{N}^+}$ be a Bernoulli measure on $\Sigma$ for the probability vector $\vect{p}=(p_1, p_2, p_3)$, and let $\nu_\ind={\pi_\ind}_*\mu$. Then
$$
\dim_H\nu_\ind =\dfrac{-\sum_{i=1}^3p_i\log p_i+\sum_{k=1}^\infty  \sum_{m=1}^{k}\binom{k-1}{m-1}p_3\log\left(\frac{m}{k}\right)(p_1^mp_2^{k-m}+p_1^{k-m}p_2^m)}{-p_1\log\alpha-p_2\log\beta-p_3\log\gamma}.
$$
\end{theorem}

\begin{proof}
	The statement follows by the combination of Theorem~\ref{086}, Lemma~\ref{083}, Lemma~\ref{lem:int} and Proposition~\ref{prop:int}.
\end{proof}

\begin{remark}\label{remark}
	An alternative proof of Theorem~\ref{thm:dimforsep} would be an application the results of Mihailescu and Urba\'nski \cite{MihUrb}. Namely, one can show that if the IFS $\mathcal{S}_{\alpha,\beta,\gamma}=\{S_1,S_2,S_3\}$ in \eqref{def:system} is forward separated then the self-similar measure $\nu'$ of the infinite IFS $\mathcal{S}'=\{S_{\ii}\}_{\ii\in\mathcal{C}}$, where
	$\mathcal{C}=\{(\overbrace{1\cdots1}^{k}\overbrace{2\cdots2}^\ell3):0\leq k,\ell\}$	with probability vector $\underline{p}'=\left(\binom{k+\ell}{k}p_1^kp_2^\ell p_3\right)_{0\leq k,\ell}$ is equivalent to the self-similar measure $\nu$ of the IFS $\mathcal{S}_\ind$ with probabilities $\underline{p}=(p_1,p_2,p_3)$. Then the  Theorem~\ref{thm:dimforsep} follows by simple algebraic manipulations and by \cite[Theorem~2.5(b), Theorem~3.11]{MihUrb} on infinite IFS.
	\end{remark}

\section{Existence of forward separated systems}\label{sec:exist}

In this section, we follow the argument of Kamalutdinov and Tetenov \cite{Tetenov}. The proof of the following theorem can be found in \cite[Theorem~14]{Tetenov}.

\begin{theorem}[General Position Theorem]\label{097} Let $(D, d_D), (L_1, d_{L_1}), (L_2, d_{L_2})$ be compact metric spaces and let $\varphi_i(\xi, x): D\times L_i \rightarrow \R^n$ for $ i\in \{ 1, 2  \}$ be continuous functions.  Suppose that

	\begin{enumerate}
		\item $\text{there exists } \alpha>0\text{ and $C>0$ such that for all } i\in \{ 1, 2 \}, \xi \in D\text{ and for all } x, y \in~L_i$
		$$
		\norm{\varphi_i(\xi, x)- \varphi_i(\xi, y)}\leq C d_{L_i}(x, y)^\alpha,
		$$
		where $\norm{\cdot}$ is the Euclidean norm in $\R^n$.
		\item Let $\Phi: D\times L_1\times L_2\rightarrow \R^n \quad \Phi(\xi, x_1, x_2)=\varphi_1(\xi, x_1)-\varphi_2(\xi, x_2)$ such that
		\begin{align}
		\begin{split}
		\text{ there exist } M>0 \text{ for all } \xi , \xi ' \in D &\text{ for all } x_1 \in L_1\text{ for all } x_2 \in L_2\\
		\norm{\Phi(\xi, x_1, x_2)-\Phi(\xi', x_1, x_2)}&\geq M d_D(\xi, \xi').
		\end{split}
		\end{align}
	\end{enumerate}
	
	Then the set $\Delta= \{\xi \in D : \varphi_1(\xi, L_1)\cap \varphi_2(\xi, L_2)\neq \emptyset     \}$ is compact in D and
	\begin{align}
	\hdim (\Delta)\leq  \frac{\hdim(L_1 \times L_2 )  }{\alpha} .
	\end{align}
\end{theorem}

\begin{lemma}[Displacement lemma]\label{094}
	Let $\mathcal{S}=\{S_1, \dots , S_m    \}$ and $\tilde{\mathcal{S}}=\{\tilde{S}_1, \dots  , \tilde{S}_m  \}$ be two iterated function systems on $\R^n$. We denote the natural projection of $\mathcal{S}$ with $\pi: \Sigma \rightarrow \R^n$ and the natural projection of $\tilde{ \mathcal{S}}$ with $\tilde{\pi}: \Sigma \rightarrow \R^n$, where $\Sigma=\{ 1, \dots, m   \}^{\mathbb{N}^+}$ is the symbolic space. Let $V\subseteq \R^n$ be a compact set such that for every $ i \in \{ 1, \dots , m\}, \, \,  S_i(V)\subseteq V$ and $\tilde{S}_i(V)\subseteq V$. Then
	\begin{align}
	\begin{split}
	\text{for all } \ii=(i_1, i_2, \dots)\in \Sigma  \quad &\norm{\pi(\ii) -\tilde{\pi}(\ii)}\leq \frac{\delta}{1-p},
	\end{split}
	\end{align}
	where
	\begin{align}
	\begin{split}
	\delta=\max \{ \norm{S_i(x)-\tilde{S}_i(x)}&: i\in \{ 1, \dots, m\} , \quad x\in V    \}\text{ and } \\
	p=\max_{1\leq i\leq m }\{   \max \{& \mathrm{Lip}(S_i) , \mathrm{Lip}(\tilde{S}_i) \}\}.
	\end{split}
	\end{align}
\end{lemma}

The proof of Lemma~\ref{094} can be found in \cite[Theorem~15]{Tetenov}.

Let us equip the symbolic space $\Sigma=\{ 1, \ldots , m   \}^{\mathbb{N}^+}$ with the metric $\rho_a$ such that $\rho_a(\ii,\jj)=a^{s(\ii, \jj)}$, where $0<a<1$ and $s(\ii, \jj)=\inf\{k\geq  0 : i_{k+1}\neq j_{k+1} \}$. It is a well known fact that the metric space $(\Sigma, \rho_a)$ is compact, moreover
\begin{align}\label{096}
\hdim(\Sigma)=\frac{\log m}{-\log a}<\frac{1}{2}.
\end{align}

The symbolic space of $\mathcal{S}_\ind $ is $\Sigma= \{ 1, 2, 3   \}^{\mathbb{N}^+}$. By the above consideration $\dim_H(\Sigma)<\frac{1}{2}$ in the metric $\rho_a$ if and only if $0<a<\frac{1}{9}$.

In the rest of this section we investigate the family of the systems $\mathcal{S}_\ind$.

\begin{lemma}\label{092}
	Let $a\in \left( 0, \frac{1}{9}  \right) $ and $\alpha, \beta, \gamma<a$. Then the natural projection $\pi_\ind\colon\Sigma\rightarrow\R$ of the system $\mathcal{S}_\ind$ is 1-Lipschitz with respect to the metric $\rho_a$ and the usual Euclidean norm on $\R$.
\end{lemma}
\begin{proof}
	Let $\ii, \jj\in \Sigma$ with $s(\ii, \jj)=k$. Then $\rho_a(\ii, \jj)=a^k$ and $i_1=j_1, \dots , i_k=j_k$, thus $\pi_\ind(\ii), \pi_\ind(\jj)\in S_{i_1\dots i_k}(\Lambda_\ind )$. The diameter of $S_{i_1\dots i_k}(\Lambda_\ind )$ is $\mathrm{Lip}(S_{i_1})\cdot \dots \cdot \mathrm{Lip}(S_{i_k})$, which is strictly smaller than $a^k$. So
$\abs{\pi_\ind(\ii)-\pi_\ind(\jj)}\leq a^k=\rho_a(\ii, \jj).$
\end{proof}

\begin{lemma}
	Let $m, n\in \mathbb{N}^+$ be arbitrary. Let $\alpha, \beta, \gamma\in \left(0, \frac{1}{9}\right)$ and consider the system $\mathcal{S}_\ind$. If $S_1^m(R_\ind)\cap S_2^n(R_\ind)\neq \emptyset$, then $\frac{8}{9}\leq \frac{\alpha^m}{\beta^n}\leq \frac{9}{8}$.
\end{lemma}
\begin{proof}
	If $\alpha, \beta, \gamma\in \left( 0, \frac{1}{9} \right)$, then $R_\ind \subseteq [\frac{8}{9}, 1]$. Thus $S_1^m(R_\ind)\subseteq [\frac{8}{9}\alpha^m, \alpha^m]$ and $S_2^n(R_\ind )\subseteq [\frac{8}{9}\beta^n, \beta^n]$. The intersection can not happen if $\alpha^m< \frac{8}{9}\beta^n $ or $\beta^n< \frac{8}{9}\alpha^m$. Thus we do not have intersection if
	$
	\frac{\alpha^m}{\beta^n}<\frac{8}{9}\text{ or }\frac{\alpha^m}{\beta^n}>\frac{9}{8}
	$. So $\frac{8}{9}\leq \frac{\alpha^m}{\beta^n}\leq \frac{9}{8}$.
\end{proof}

\begin{lemma}\label{091}
	Let $m, n \in \mathbb{N}^+$ and $\beta, \gamma \in (0, \frac{1}{9})$ be fixed. Denote
	\begin{align}
	D_{m, n}(\beta, \gamma)=\left\lbrace \alpha\in \left(0, \frac{1}{9}\right): \frac{8}{9}\leq \frac{\alpha^m}{\beta^n}\leq \frac{9}{8}   \right\rbrace.
	\end{align}
	Let us define for $\alpha\in D_{m,n}(\beta,\gamma)$ and $\ii\in\Sigma$
	\begin{align*}
	\begin{split}
	 \varphi_1(\alpha, \ii)&=\pi_\ind ((1)^m*(3)*\ii)=S_1^mS_3(\pi_\ind(\ii)), \\
	\varphi_2(\alpha, \ii)&=\pi_\ind ((2)^n*(3)*\ii)=S_2^nS_3(\pi_\ind (\ii)).
	\end{split}
	\end{align*}
	Then for every $\alpha, \alpha'\in D_{m, n }(\beta, \gamma)$ and for every $\ii, \jj\in \Sigma$
	\begin{align}
	\abs{\varphi_1(\alpha, \ii)-\varphi_2(\alpha, \jj)-\varphi_1(\alpha', \ii)+\varphi_2(\alpha', \jj)}\geq M\abs{\alpha-\alpha'},
	\end{align}
	where $M(m, n, \beta, \gamma)>0$ constant.
\end{lemma}
\begin{proof}
	Let $\alpha, \alpha'\in D_{m, n}(\beta, \gamma)$ and $\ii, \jj\in \Sigma$ be arbitrary. We introduce the notation $\mathcal{S}=\mathcal{S}_\ind=\{ S_1, S_2, S_3\}$,  $\mathcal{S}'=\mathcal{S}_{\alpha', \beta, \gamma}=\{ S_1', S_2', S_3'\} $, let $\pi=\pi_\ind$ and $\pi'=\pi_{\alpha', \beta, \gamma}$. Observe $S_2'=S_2$ and $S_3'=S_3$.
	\par
	Let $\alpha<\alpha'$ and $\delta=\abs{\alpha'-\alpha}$, then using Lagrange mean value theorem
	\begin{align}\label{095}
	m\alpha^{m-1}\leq \frac{\alpha'^m-\alpha^m}{\alpha'-\alpha}=\frac{\abs{\alpha'^m-\alpha^m}}{\delta}\leq m\alpha'^{m-1}.
	\end{align}
	
	We defined $\delta=\abs{\alpha'-\alpha}$ and using Lemma~\ref{094} for $\mathcal{S}$ and $\mathcal{S}'$, then we get
	\begin{align}\label{093}
	\text{ for every } \ii\in \Sigma\quad \abs{\pi(\ii)-\pi'(\ii)}\leq \frac{9}{8}\delta.
	\end{align}
	Consider the difference that we have to estimate
	\begin{align*}
	\begin{split}
	\varphi_1(\alpha, \ii)&-\varphi_1(\alpha', \ii)+\varphi_2(\alpha', \jj)-\varphi_2(\alpha, \jj)=\\
	=S_1^mS_3(\pi(\ii))&-S_1'^mS_3'(\pi'(\ii))+S_2'^nS_3'(\pi'(\jj))-S_2^nS_3(\pi(\jj))=\\
	=S_1^mS_3(\pi(\ii))&-S_1'^mS_3(\pi'(\ii))+S_2^nS_3(\pi'(\jj))-S_2^nS_3(\pi(\jj))=\\
	= \underbrace{S_1^mS_3(\pi(\ii))-S_1^mS_3(\pi'(\ii))}_A&+\underbrace{S_1^mS_3(\pi'(\ii))-S_1'^mS_3(\pi'(\ii))}_B+\underbrace{S_2^nS_3(\pi'(\jj))-S_2^nS_3(\pi(\jj))}_C .
	\end{split}
	\end{align*}
	We will use the estimate
	\begin{align}
	\abs{A+B+C}\geq \abs{B}-\abs{A+C}\geq \abs{B}-\abs{A}-\abs{C}.
	\end{align}
	Consider $\abs{A}$ part of the above calculation
	\begin{align*}
	\abs{A}=\abs{S_1^mS_3(\pi(\ii))-S_1^mS_3(\pi'(\ii))}=\alpha^m \gamma\abs{\pi(\ii)-\pi'(\ii)}\leq \frac{9}{8}\alpha^m \gamma \delta,
	\end{align*}
	where in the inequation we use \eqref{093}. The next part is
	\begin{align*}
	\abs{B}=\abs{S_1^mS_3(\pi'(\ii))-S_1'^mS_3(\pi'(\ii))}=\abs{\alpha^m-\alpha'^m}\abs{S_3(\pi'(\ii))}\geq \frac{8}{9} m\alpha^{m-1}\delta,
	\end{align*}
	where in the inequation we use \eqref{095}. Finally,
	\begin{align*}
	\abs{C}=\abs{S_2^nS_3(\pi'(\jj))-S_2^nS_3(\pi(\jj))}=\beta^n\gamma\abs{\pi(\jj)-\pi'(\jj)}\leq \frac{9}{8}\beta^n\gamma\delta,
	\end{align*}
	where we used again \eqref{093}.
	\par
	Now estimate
	\begin{align}
	\abs{B}-\abs{A}\geq \left( \frac{8m}{9\alpha}-\frac{9}{8}\gamma  \right)\alpha^m\delta\geq  \left( 8-\frac{9}{8} \right)\alpha^m\delta\geq \left( 8-\frac{9}{8} \right)\frac{8}{9}\beta^n\delta>6\beta^n \delta,
	\end{align}
	where in the second inequation we use $\gamma<1, \, m\geq 1, \, \alpha<\frac{1}{9}$.
	\par
	The following
	\begin{align}
	\abs{C}\leq \frac{9}{8}\gamma\beta^n \delta< \beta^n\delta
	\end{align}
	is true, beacuse $\gamma<\frac{1}{9}$.
	Thus
	\begin{align}
	\abs{\varphi_1(\alpha, \ii)-\varphi_2(\alpha, \jj)-\varphi_1(\alpha', \ii)+\varphi_2(\alpha', \jj)}\geq 5\beta^n\abs{\alpha'-\alpha},
	\end{align}
	so $M=5\beta^n$.
\end{proof}

\begin{theorem}\label{090}
	Let $\beta, \gamma\in (0, \frac{1}{9})$. Then
	\begin{align}
	\leb\left((0,\beta)\setminus\left\lbrace \alpha\in (0,\beta): \mathcal{S}_\ind \text{ is a forward separated system}   \right\rbrace\right)=0.
	\end{align}
\end{theorem}

\begin{proof}
	Let $m, n \in \mathbb{N}^+$ and $\beta,\gamma\in(0,1/9)$ be arbitrary but fixed. First, we show that for the set
	\begin{align}
	\Delta_{m , n }(\beta, \gamma)=\left\lbrace \alpha\in \left(0, \frac{1}{9}\right): S_1^m(R_\ind)\cap S_2^n(R_\ind )    \neq \emptyset \right\rbrace
	\end{align}
	 $\mathcal{L}(\Delta_{m, n}(\beta, \gamma))=0$.
	
	Let $\varepsilon>0$ be such that $\frac{1}{9}-\varepsilon>\beta, \gamma$. Then $E_{m, n}(\beta, \gamma)=D_{m, n}(\beta, \gamma)\cap [\varepsilon, \frac{1}{9}-\varepsilon]$ is a closed interval in $\R$, so it is compact. We consider the compact metric space $(\Sigma, \rho_{a})$, where $\Sigma=\{  1, 2, 3 \}^{\mathbb{N}^+}$ and $a=\frac{1}{9}-\frac{\varepsilon}{2}$.
	\par
	Let $\varphi_i: E_{m, n}(\beta, \gamma)\times \Sigma \rightarrow \R\, $ for $i=1, 2$ as in Lemma~\ref{091}. Let
	\begin{align*}
	\Xi_{m, n}^\varepsilon(\beta, \gamma)=\Delta_{m , n }(\beta, \gamma)\cap \left[ \varepsilon, \frac{1}{9}-\varepsilon\right].
	\end{align*}
	For an $\alpha$ the $S_1^m(R_\ind)\cap S_2^n(R_\ind)\neq \emptyset$ holds if and only if there exist $ \ii, \jj\in \Sigma$ such that $\varphi_1(\alpha, \ii)=\varphi_2(\alpha, \jj)$, thus
	\begin{align}
	\Xi_{m, n}^\varepsilon(\beta, \gamma)=\{ \alpha\in E_{m, n}(\beta, \gamma):\varphi_1(\alpha, \Sigma)\cap \varphi_2(\alpha, \Sigma)\neq \emptyset    \}.
	\end{align}
	Using Lemma \ref{092} one can see that $\varphi_i$ is H\"older continuous with respect to the second variable for $i=1, 2$. Applying Lemma \ref{091}, we get that the conditions of the General Position Theorem \ref{097} holds. Using General Position Theorem \ref{097}, then get
	\begin{align}
	\hdim(\Xi_{m, n}^\varepsilon(\beta, \gamma))\leq \hdim(\Sigma\times \Sigma)\leq 2\hdim(\Sigma)<1,
	\end{align}
	the last inequality holds because of \eqref{096}. So $\leb(\Xi_{m,n}^\varepsilon(\beta, \gamma))=0$. Moreover,
	\begin{align}
	\Delta_{m, n }(\beta, \gamma)=\bigcup_{k=1}^\infty \Xi_{m, n }^{1/k}(\beta, \gamma),
	\end{align}
	thus the continuity of measure yields that $\leb(\Delta_{m, n }(\beta, \gamma))=0$.

	Finally,
	\begin{align}
	(0,\beta)\setminus\left\lbrace \alpha\in (0,\beta): \mathcal{S}_\ind \text{ is a forward separated system}   \right\rbrace=\bigcup_{m, n=1}^\infty \Delta_{m, n}(\beta,\gamma),
\end{align}
and thus, the statement follows.
\end{proof}

Finally, Theorem~\ref{thm:main} follows by Theorem~\ref{thm:dimforsep} and Theorem~\ref{090}.

\bibliographystyle{plain}
\bibliography{biblio}

\end{document}